\title{The Order of the Non-universal Tree of a Hilbert Space with Respect to the Haar Basis}
\author{Sam Whitmire
	 \thanks{Indiana University Bloomington Department of Mathematics, Bloomington, IN 47405}
                  }
\documentclass{article}
\usepackage{graphicx}
\usepackage{amssymb}
\usepackage{amsmath}
\usepackage{mathtools}
\usepackage{amsthm}
\usepackage{geometry}
\usepackage{biblatex}
\newtheorem{theorem}{Theorem}
\newtheorem{lemma}[theorem]{Lemma}
\newtheorem{proposition}[theorem]{Proposition}
\newtheorem{definition}[theorem]{Definition}
\newtheorem{corollary}[theorem]{Corollary}

\newcommand{\spn}{\text{span}}
\newcommand{\N}{\mathbb N}
\newcommand{\R}{\mathbb R}

\newcommand{\floor}[1]{\lfloor #1 \rfloor}

\begin{document}
\newpage
\maketitle
%%%%%%%%%%%%%%%%%%%%%%%%%%%
% abstract, keywords and Subject classification are optional.
%%%%%%%%%%%%%%%%%%%%%%%%%%%
\begin{abstract}
In \cite{Bo}, Bossard introduced the notion of the non-universal tree $T_{NU}(X)$ associated with each separable Banach space $X$ which does not contain an isomorphic copy of $C(2^\N)$. Together with the order operation defined on well-founded trees, we obtain a method of classifying the complexity of separable Banach spaces by the degree of isomorphism of finite-dimensional subspaces of $C(2^\N)$. Despite further refinement of this concept in later years, we are unaware of any specific instances of direct computations of the order of the non-universal tree for a concrete space and basis. We show that if $H$ is any separable Hilbert space, then $o(T_{NU}(H)) = \omega + 1$ when taken with respect to the Haar basis for $C(2^\N)$, demonstrating that the class of Hilbert spaces is the least complex class with respect to this measurement when considering this basis.
\end{abstract}

% Most people don't use these, so they are "commented out"
% by starting the lines with a "%"
%\begin{keywords}
%   \LaTeX, typesetting
%\end{keywords}

%\begin{AMS}
%   50C60, 18C25
%\end{AMS}

%%%%%%%%%%%%%%%%%%%%%%
% % Here is the start of the Text
%%%%%%%%%%%%%%%%%%%%%%
\section{Introduction}
The conceptualization of the construction of the NU tree for a space $X$ can be traced back to Bourgain in \cite{Bu}. In \cite{Bo}, Bossard formalized the development and demonstrated that $X\mapsto o(T_{NC}(X))$ is a $\boldsymbol{\Pi}_1^1$-rank on $NU$, the collection of all separable Banach spaces which do not contain an isomorphic copy of $C(2^\N)$. At each step of the construction, one asks for an isomorphic equivalence between the spaces spanned by finite sequences in $X$ and finite lists of elements of a Schauder basis of $C(2^\N)$; hence this order operation measures the degree to which $X$ contains isomorphic finite-dimensional subspaces of $C(2^\N)$ despite not containing an isomorphic copy of $C(2^\N)$ itself.

Much interest has been expressed regarding this mapping. For example, in \cite{AD}, the question of dependence of $o(T_{NU}(X))$ on the choice of basis used in the construction was tackled, and they were able to verify that there is a $\boldsymbol\Pi_1^1$ rank on NU which uniformly bounds $o(T_{NU}(X,(e_n)))$ for each normalized Schauder basis $(e_n)$ of $C(2^\N)$. Even so, to our knowledge, there are no concrete examples of actual computations of $o(T_{NU}(X,(e_n)))$ for any space $X$ and any basis $(e_n)$ of $C(2^\N)$. This paper aims to fill that gap with the following result:

\begin{theorem}\label{MR}
Let $H$ be a separable Hilbert space and $(f_{t_n})_{n=0}^\infty$ be the Haar basis for $C(2^\N)$. Then $\\o(T_{NU}(X,(f_{t_n}))) = \omega + 1$.
\end{theorem}

We also prove that $o(T_{NU}(X))\geq\omega + 1$ for any infinite-dimensional separable Banach space $X$, so this result indicates that Hilbert spaces lie at the absolute minimum of the complexity spectrum of Banach spaces which do not contain $C(2^\N)$ and for which comparison is made with the Haar basis. In certain cases, we are able to prove direct estimates on the lengths of the nodes contained in the tree, and we derive these in the penultimate section of the paper.

\section{Preliminaries}
In this section, we present a rapid review of the fundamental concepts required to discuss the construction of the tree in Theorem \ref{MR}. For a more detailed exposition, we invite the reader to consult \cite{K}. 

The construction relies on a series of ideas united from functional analysis and descriptive set theory. By $2^\N$ we mean the set of all binary sequences under the metric 
\begin{equation*}
d(\alpha,\beta) = \begin{cases}
\frac{1}{n+1} & n \text{ is the least integer for which $\alpha(n)\neq\beta(n)$}\\
0 & \alpha = \beta
\end{cases}
\end{equation*}
It is well-known that $2^\N$ is a compact zero-dimensional Polish space under this metric. By $C(2^\N)$ we mean the set of all real-valued continuous functions on $2^\N$ equipped with the usual supremum norm
\begin{equation*}
||f||_\infty = \sup_{\alpha\in 2^\N} \{|f(\alpha)|\}
\end{equation*}
which makes $C(2^\N)$ into a Banach space. This space has a natural Schauder basis called the Haar basis which is a key component of Theorem \ref{MR}. Let $2^{<\N}$ be the set of all finite binary sequences and let $\mathcal O\subseteq 2^{<\N}$ consist of the empty string together with all binary strings which end in 0. For $s,t\in \mathcal O$, define $s\cap t$ to be the longest string which is an initial segment of both $s$ and $t$. Also define $s \prec t$ if and only if $|s| < |t|$ or [$|s| = |t|$ and at the first index $i$ for which $s(i)\neq t(i)$, we have $s(i) = 0$ and $t(i) = 1$].

With a bit of effort, we have the following proposition (pp. 109-110 in \cite{D}):

\begin{proposition}
There is a unique bijection $h : \mathcal O\rightarrow\N$ which satisfies the following properties:
\begin{enumerate}
\item If $|s| < |t|$, then $h(s) < h(t)$.
\item If $|s|=|t|$ and $s \prec t$, then $h(s) < h(t)$.
\end{enumerate}
\end{proposition}

By setting $t_n = h^{-1}(n)$, we obtain an enumeration of the strings in $\mathcal O$. Here is a sample listing of the first few values of $h$:
\begin{align*}
&h(\varnothing) = 0 \\
& h((0)) = 1 \\
&h((00)) = 2 \qquad  h((10)) = 3 \\ 
&h((000)) = 4 \qquad h((010))=5 \qquad h((100)) = 6 \qquad h((110))=7 \qquad \\
&\ldots
\end{align*}
Now let $V_n = \{\alpha\in 2^\N\mid \alpha|_{|t_n|} = t_n\}$. This is a clopen subset of $2^\N$, and so its characteristic function $\chi_{V_n} : 2^\N\rightarrow\R$ is continuous. Set $f_{t_n} = \chi_{V_n}$. One can show that $(f_{t_n})$ forms a monotone Schauder basis of $C(2^\N)$ (cf. Claim 6.25(i) in \cite{D}).

Let $A$ be a countable set. The set of all finite sequences in $A$ is denoted $A^{<\N}$. For two strings $s,t\in A^{<\N}$, we write $s\sqsubseteq t$ if $s$ is an inital substring of $t$. A tree over $A$ is a set $T\subseteq A^{<\N}$ which is downwards closed, i.e. if $t\in T$ and $s\sqsubseteq t$, then $s\in T$ as well. $T$ is said to be ill-founded if there is a sequence $\alpha = (a_n)\in A^\N$ with $\alpha|_k\in T$ for every $k\in\N$, and is said to be well-founded otherwise. Let $Tr_A = \{T\subseteq A^{\N}\mid T\text{ is a tree}\}$ and $WF_A = \{T\in Tr_A\mid T\text{ is well-founded}\}$. For a well-founded tree $T$, we can define the iterated derivatives $\{T^{(\xi)}\mid \xi<\omega_1\}$ of $T$ by the rules
\begin{align*}
&T^{(0)} = T\\
&T^{(\xi+1)} = \{t\in T^{(\xi)}\mid (\exists s\in T)[s\sqsupset t]\}\\
&T^{(\lambda)} = \bigcap_{\xi < \lambda} T^{(\xi)} \text{ when $\lambda$ is a limit ordinal}
\end{align*}
We think of the successor derivative operation as trimming the furthest leaves off of each branch of our tree, while the limit derivative operation consolidates all of these trimmings together. When $T$ is well-founded, there must exist $\xi < \omega_1$ with $T^{(\xi)} = \varnothing$; the smallest such $\xi$ is called the order of the tree and we write $o(T) = \xi$. There is also a phrasing of this operation in terms of well-founded relations, from which it follows that the order of every well-founded tree is a successor ordinal (cf. pp. 10-11 in \cite{K}).

For a Polish space $X$ and $W\subseteq X$, $W$ is called analytic (written $W\in\boldsymbol\Sigma_1^1(X)$) if $W = \varnothing$ or $W$ is the image of $\N^\N$ under a continuous function. $W$ is called coanalytic (written $W\in\boldsymbol\Pi_1^1(X)$) if $X\setminus W$ is analytic. A function $\varphi:C\rightarrow\omega_1$ on $C\in\boldsymbol\Pi_1^1(X)$ is called a $\boldsymbol\Pi_1^1$-rank on $X$ if there are relations $\leq_\Sigma\in \boldsymbol\Sigma_1^1(X\times X), \leq_\Pi\in\boldsymbol\Pi_1^1(X\times X)$ such that for each $y\in C$,
\begin{equation*}
\{x\in C\mid \varphi(x)\leq\varphi(y)\} = \{x\in X\mid x\leq_\Sigma y\} = \{x\in X\mid x\leq_\Pi y\}
\end{equation*}

The following result is standard and can be found on any text on descriptive set theory, e.g. Theorem 27.1 and Exercise 34.6(1) in \cite{K}.

\begin{proposition}
$WF_A\in\boldsymbol\Pi_1^1(Tr_A)$, and for $T\in WF_A$, $T\mapsto o(T)$ is a $\boldsymbol\Pi_1^1$ rank on $WF_A$.
\end{proposition}

Now we discuss the construction of the trees. Let $X$ be a separable Banach space.  Fix $C\in[1,\infty)$ and a separable Banach space $Y$. Given $m\in\N$, two sequences $(x_n)_{n=0}^{m-1}\subseteq X, (y_n)_{n=0}^{m-1}\subseteq Y$ are said to be $C$-equivalent if for each sequence $(a_n)_{n=0}^{m-1}\subseteq\R$, we have
\begin{equation*}
\frac{1}{C}\bigg|\bigg|\sum_{n=0}^{m-1} a_n x_n\bigg|\bigg|\leq \bigg|\bigg|\sum_{n=0}^{m-1} a_n y_n\bigg|\bigg|\leq C\bigg|\bigg|\sum_{n=0}^{m-1} a_n x_n\bigg|\bigg|
\end{equation*}
Clearly $C$-equivalence is a symmetric condition. Throughout, we will assume that the space $Y$ of interest for the above definition is $C(2^\N)$, and we will also assume that the sequence $(y_n)$ is an initial segment of the Haar basis $(f_{t_n})_{n\in\N}$ for $C(2^\N)$ defined above. Define the tree $\textbf{T}_{NU}(X,(f_{t_n}),C)$ on $X$ by the rule $(x_n)_{n=0}^{m-1}\in \textbf{T}_{NU}(X,(f_{t_n}),C)$ if and only if $(x_n)_{n=0}^{m-1}$ is $C$-equivalent to $(f_{t_n})_{n=0}^{m-1}$. The following fact due to Bourgain is straightforward to verify (see, for instance, \cite{D}, Lemma 2.16):

\begin{theorem}\label{WFR}
$X$ does not contain a copy of $C(2^\N)$ if and only if for every $C\in [1,\infty)$, $\textbf{T}_{NU}(X,(f_{t_n}),C)$ is well-founded.
\end{theorem}

The theory of well-founded trees is substantially more developed for $\N$ than it is for Banach spaces, so we seek to convert this into a tree on $\N$ rather than $X$. Moreover, to apply this theory, we need this transformation to occur in a Borel manner. To do so, we utilize the universal space for separable Banach spaces. It is well known that if $Z$ is any separable Banach space, then $Z$ is isometrically isomorphic to a closed subspace of $C(2^\N)$. When endowed with the Effros Borel structure inherited from the set of all closed \textit{subsets} of $C(2^\N)$, the collection of all closed \textit{subspaces} of $C(2^\N)$ becomes a standard Borel space which contains an isometric copy of every separable Banach space. We denote this standard Borel space by SB. (This construction lies at the heart of a great number of advancements in functional analysis; more information about the basic facts can be found in Chapters 1 and 2 of \cite{D}.) By the Kuratowski-Ryll-Nardzewski selection theorem (see e.g. \cite{K}, Theorem 12.13), there are Borel functions $\{d_n : \text{SB}\setminus\{\varnothing\}\rightarrow C(2^\N)\}$ such that for every $X\in \text{SB}\setminus\{\varnothing\}$, $\{d_n(X)\}_{n\in\N}$ is dense in $X$. Using these, we define a tree $T_{NU}(X,(f_{t_n}),k)$ ($k\in\N$ with $k\geq 1$) on $\N$ by the rule $(\ell_0,\ldots,\ell_{m-1})\in T_{NU}(X,(f_{t_n}),k)$ if and only if $(d_{\ell_n}(X))_{n=0}^{m-1}$ is $k$-equivalent to $(f_{t_n})_{n=0}^{m-1}$. It is easily verifiable that an analogue to Theorem \ref{WFR} holds for this new countable collection of trees. Finally, to condense them into one tree, we define the tree $T_{NU}(X) = T_{NU}(X,(f_{t_n}))$ on $\N$ by the rule $(k,\ell_0,\ldots, \ell_{m-1})\in T_{NU}(X)$ if and only if $(\ell_0,\ldots, \ell_{m-1})\in T_{NU}(X,(f_{t_n}), k)$. Note that vacuously we have $\varnothing\in T$ and $(k)\in T$ for every $k\in\N$ with $k\geq 1$, which ensures that $T$ is genuinely a tree. $T_{NU}(X)$ is called the non-universal tree of $X$ with respect to the Haar basis. This is a Borel construction in the sense that the map $X\mapsto T_{NU}(X,(f_{t_n}))$ is a Borel map (following the argument presented in Lemma 2.4 in \cite{D}), and since $\text{WF}$ is a $\boldsymbol\Pi_1^1$ subset of $\text{Tr}$, we have the following result due to Bossard (c.f. Theorem 4.4 in \cite{Bo}):

\begin{theorem}\label{WFN}
The set $\text{NU} = \{X\in\text{SB}\mid X\text{ does not contain an isomorphic copy of $C(2^\N)$}\}$ is a $\boldsymbol\Pi_1^1$ subset of SB, and the map $X\mapsto o(T_{NU}(X))$ is a $\boldsymbol\Pi_1^1$-rank on NU.
\end{theorem}

\section{Main Result}

To prove Theorem \ref{MR}, our first step is to illustrate the result for $\ell^2$. We will then argue that the result persists for every other Hilbert space $H$ from the $\ell^2$ case.

We start with the lower bound in the theorem. In general, one cannot say much about the order of the non-universal tree of an arbitrary separable Banach space beyond the following simple estimate, but in our case, it turns out to be sufficient. (It is worth noting that this theorem is true regardless of the Schauder basis used to construct the tree $T_{NU}$.)

\begin{theorem}\label{ID}
Let $X\in \text{NU}$ be infinite-dimensional. Then $o(T_{NU}(X))\geq\omega + 1$.
\end{theorem}

\begin{proof}
Since $X$ is infinite-dimensional, $X$ must have an infinite linearly independent subset of vectors, and by density, we can uncover such a family by taking a subsequence of the vectors $\{d_\ell(X)\}_{\ell=0}^\infty$ obtained from the selection functions. Let $(d_{\ell_n}(X))_{n=0}^\infty$ be linearly independent. Fix $m$ and consider the finite sequences $(d_{\ell_n}(X))_{n=0}^{m-1}$ and $(f_{t_n})_{n=0}^{m-1}$. Let $T : \overline{\spn\{d_{\ell_0}(X),\ldots, d_{\ell_{m-1}}(X)\}}\rightarrow \overline{\spn\{f_{t_0},\ldots, f_{t_{m-1}}\}}$ be defined by $T(d_{\ell_n}(X)) = f_{t_n}$. This is a linear operator between finite dimensional spaces and hence must be continuous. Also, $T$ is invertible, so $T^{-1}$ is also continuous. Hence there exists a number $k_m\in\N$ such that for every $(a_n)_{n=0}^{m-1}\subseteq \R$, we have
\begin{equation*}
\frac{1}{k_m}\bigg|\bigg|\sum_{n=0}^{m-1} a_n d_{\ell_n}(X)\bigg|\bigg|\leq \bigg|\bigg|\sum_{n=0}^{m-1} a_n f_{t_n}\bigg|\bigg|_\infty \leq k_m\bigg|\bigg|\sum_{n=0}^{m-1} a_n d_{\ell_n}(X)\bigg|\bigg|
\end{equation*}
which implies that $(k_m, \ell_0, \ldots, \ell_{m-1})\in T_{NU}(X)$ and $o(T_{NU}(X))\geq m+1$. Since this holds for every $m$, $o(T_{NU}(X))\geq\omega$. Finally, as mentioned earlier, $o(T)$ must be a successor ordinal, and so $o(T_{NU}(X))\geq\omega + 1$.
\end{proof}

To find the corresponding upper bound, we introduce a new piece of terminology.

\begin{definition}
Let $T$ be a well-founded tree on a set $A$. An \textbf{infinite hand} in $T$ is a sequence $(s_n)_{n=0}^\infty\subseteq T$ with the following properties:
\begin{enumerate}
\item $(|s_n|)_{n=0}^\infty\subseteq\N$ is strictly monotone increasing.
\item There exists an element $s\in T$ with $s\sqsubset s_n$ for every $n$; the element of largest length which satisfies this property is called the \textbf{palm} of the infinite hand.
\end{enumerate}
\end{definition}

\begin{lemma}\label{Hand}
Let $T$ be a well-founded tree on a set $A$ with $o(T)\geq\omega+1$. Then $o(T) = \omega+1$ if and only if for every infinite hand in $T$, the palm of that infinite hand is $\varnothing$.
\end{lemma}

\begin{proof}
If $o(T) = \omega+1$ and $T$ contains an infinite hand with palm $s\neq\varnothing$, then $s\in T^\omega = \bigcap_{n\in\omega} T^n$ so that $o(T) > \omega+1$, which is a contradiction. Conversely, if $o(T) > \omega+1$, then there is an element $t\in T^{\omega+1}$. By taking any $s\in T^\omega$ with $s\sqsupset t$, we find a nonempty element $s\in T^\omega$. Fix $n\geq 1$. Since $s\in T^n$, there is an extension $s_n^1$ of $s$ in $T^{n-1}$ with $|s_n^1| = |s|+1$; then we can find an extension $s_n^2\in T^{n-2}$ of $s_n^1$ with $|s_n^2| = |s|+2$; iterating this procedure produces a string $s_n\in T^0 = T$ with $|s_n| = |s| + n$ and $s\sqsubset s_n$. The family $(s_n)_{n=0}^\infty$ is an infinite hand in $T$ with palm $s$.
\end{proof}

The following lemma is the keystone of the argument Theorem \ref{MR}.

\begin{lemma}\label{ML}
Let $C\geq 1$. There exists a number $N\in\N$ depending on $C$ such that if $t\in\textbf{T}_{NU}(\ell^2, (f_{t_n}),C)$, then $|t|\leq N$. In particular, there is no infinite hand in $\textbf{T}_{NU}(\ell^2, (f_{t_n}),C)$.
\end{lemma}

\begin{proof}
For notational convenience, let $\textbf{T}(C) = \textbf{T}_{NU}(\ell^2, (f_{t_n}), C)$. Suppose that $(t_i)_{i=0}^\infty$ is an infinite hand in $\textbf{T}(C)$. Write $t_i = (v_0^i,\ldots, v_{m_i-1}^i)$. Without loss of generality, we may assume $m_i = i+1$ for every $i$ (i.e. $|t_i|=i+1$ for every $i$). We will produce the quantity $N$ by looking at a specific subsequence of the $v_n^i$ vectors.

By assumption, $t_i\in \textbf{T}(C)$ for every $i$. This provides us with the inequalities
\begin{equation}\label{r1}
\frac{1}{C}\bigg|\bigg|\sum_{n=0}^{i} a_n f_{t_n}\bigg|\bigg|_\infty\leq \bigg|\bigg|\sum_{n=0}^{i} a_n v_n^i\bigg|\bigg|_2\leq C\bigg|\bigg|\sum_{n=0}^{i} a_n f_{t_n}\bigg|\bigg|_\infty
\end{equation}
for each sequence $(a_n)_{n=0}^i\in \R^{i+1}$. Observe that for $n\in\{2^k-1\mid k\in\N\setminus\{0\}\}$, the characteristic functions $\{f_{t_n}\}$ have disjoint support. Hence
\begin{equation*}
\bigg|\bigg|\sum_{n\in\{2^k-1\}, n\leq i} a_n f_{t_n}\bigg|\bigg|_{\infty} = \sup_n\{|a_n|\} = ||(a_n)||_\infty
\end{equation*}
For each $i$, assume the coefficients $(a_n)_{n=0}^i$ are zero unless $n = 2^k - 1$ for some $k\geq 1$. For simplicity, let us relabel the vectors $\{v_n^i\mid n = 2^k - 1, n\leq i\}$ as $(v_k^i)_{k=1}^{p_i}$, where $p_i = \floor{\log_2(i+1)}$. We translate the above inequality into the language of operator theory using three different operators.

Let $(u_k^i)_{k=1}^{p_i}$ be an orthonormal basis for $\overline{\spn\{v_k^i\}}$ and $S\in B(\overline{\spn\{v_k^i\}},\overline{\spn\{u_k^i\}})$ be the linear operator defined by $Sv_k^i = u_k^i$. Orthonormality implies that $\overline{\spn\{u_k^i\}}$ is isometrically isomorphic to $\ell_{p_i}^2$ via a unitary operator $U\in B(\overline{\spn\{u_k^i\}},\ell_{p_i}^2)$ with $Uu_k^i = e_k\in\ell_{p_i}^2$. Define $I_{\infty,2}\in B(\ell_{p_i}^\infty, \ell_{p_i}^2)$ as the operator which sends $e_k\in\ell_{p_i}^\infty$ to $e_k\in\ell_{p_i}^2$. Define $T\in B(\ell_{p_i}^\infty,\ell_{p_i}^2)$ by $T = US^{-1}U^{-1}I_{\infty,2}$ and note that $T$ is an isomorphism since each constituent operator in $T$ is an isomorphism. Then we have
\begin{align*}
\bigg|\bigg|T\bigg(\sum_{k=1}^{p_i} a_k e_k\bigg)\bigg|\bigg|_2 &= \bigg|\bigg|US^{-1}U^{-1}I_{\infty,2}\bigg(\sum_{k=1}^{p_i} a_k e_k\bigg)\bigg|\bigg|_2\\
&= \bigg|\bigg|S^{-1}\bigg(\sum_{k=1}^{p_i} a_k u_k^i\bigg)\bigg|\bigg|_2\\
&= \bigg|\bigg|\sum_{k=1}^{p_i} a_k v_k^i\bigg|\bigg|_2
\end{align*}
Inequality \ref{r1} becomes
\begin{equation}\label{r2}
\frac{1}{C}||(a_k)_{k=1}^{p_i}||_\infty\leq \bigg|\bigg|T\bigg(\sum_{k=1}^{p_i} a_k e_k\bigg)\bigg|\bigg|_2\leq C||(a_k)_{k=1}^{p_i}||_\infty
\end{equation}
which illustrates that $||T||\leq C$ and $||T^{-1}||\leq C$, so that $||T||\cdot||T^{-1}||\leq C^2$.

We now apply a classical theorem recorded by Tomczak-Jaegermann (Proposition 37.6 in \cite{TJ}) regarding the Banach-Mazur distance of the spaces $\ell^2_m$ and $\ell^\infty_m$, which states that
\begin{equation*}
\inf\{||S||\cdot||S^{-1}||\mid S \in B(\ell^2_m,\ell^\infty_m) \text{ is an isomorphism}\} = \sqrt{m}
\end{equation*}
Hence we have
\begin{equation*}
\sqrt{p_i} = \inf\{||S||\cdot||S^{-1}||\mid S \in B(\ell^2_{p_i},\ell^\infty_{p_i}) \text{ is an isomorphism}\} \leq ||T||\cdot||T^{-1}|| \leq C^2
\end{equation*}
Now $p_i$ is increasing and tends to infinity as $i\rightarrow\infty$, which contradicts this inequality. Therefore inequality \ref{r1} can only be satisfied when $p_i\leq C^4$; taking $N \geq 2^{\floor{C^4}+1}-1$ provides the stated upper bound.
\end{proof}

To verify Theorem \ref{MR} using this lemma, we climb the tower of the tree construction. We have just shown that for any $C\geq 1$, the trees $\textbf{T}_{\text{NU}}(\ell^2, (f_{t_n}), C)$ from the first step of the construction are uniformly bounded: there exists $N\in\N$ such that if $t\in \textbf{T}_{NU}(\ell^2,(f_{t_n}),C)$, then $|t|\leq N$. This holds regardless of the choice of vectors in the string $t$. The trees in the second step are defined by $(\ell_0,\ldots, \ell_{m-1})\in T_{NU}(\ell^2, (f_{t_n}), k)$ if and only if $(d_{\ell_n}(\ell^2))_{n=0}^{m-1}\in \textbf{T}_{NU}(\ell^2,(f_{t_n}),k)$, and hence there is an $N'$ such that if $t\in T_{NU}(\ell^2, (f_{t_n}), k)$, then $|t|\leq N'$. Finally, the ``real tree'' $T_{NU}(X,(f_{t_n}))$ is defined by amalgamating these second-step trees: $(k,\ell_0,\ldots, \ell_{m-1})\in T_{NU}(\ell^2, (f_{t_n}))$ if and only if $(\ell_0,\ldots, \ell_{m-1})\in T_{NU}(\ell^2, (f_{t_n}),k)$.

By Theorem \ref{ID}, we have $o(T_{NU}(\ell^2,(f_{t_n})))\geq\omega+1$, and by Lemma \ref{Hand}, $o(T_{NU}(\ell^2,(f_{t_n})))> \omega+1$ if and only if there is an infinite hand in $T_{NU}(\ell^2,(f_{t_n}))$ with palm $s\neq\varnothing$. If such a hand and palm were to exist, then there would exist $(s_n)_{n=0}^\infty\subseteq T_{NU}(\ell^2,(f_{t_n}))$ with $|s_n|\rightarrow\infty$ monotonically and $s\sqsubset s_n$ for all $n$. If $s = (k, \ell_0,\ldots, \ell_{m-1})$, then there is a constant $N'$ for which if $s'\sqsupset s$, then $|s'|\leq N'$, which contradicts the properties of the sequence $(s_n)$. The same is true for $s = (k)$. All nonempty elements in $T_{NU}(\ell^2,(f_{t_n}))$ are of this form, so no such hand and palm can exist, and with this we have proved Theorem 1 for $H = \ell^2$.

Finally, we expand the result to an arbitrary separable Hilbert space $H$. The definitions of each of the classes of trees in the construction of $T_{NU}(H)$ follow rules related to the satisfaction of norm inequalities by arbitrary linear combinations of vectors. As a result, since $H$ and $\ell^2$ are isometrically isomorphic, the uniform bounds from Lemma \ref{ML} apply to $H$ as well. Following the argument in the preceding two paragraphs with this information in mind proves Theorem \ref{MR} for every separable Hilbert space $H$.

\section{Improved Estimates}

This section contains some stronger estimates on the lengths of the strings of vectors in $\ell^2$ under certain additional conditions. Our starting point is the map $S_m : \ell^2_m\rightarrow \spn\{f_{t_0},\ldots, f_{t_{m-1}}\}$ ($m\geq 1$) defined by $S(e_n) = f_{t_n}$. We would like to determine the operator norm of $S_m$. Observe that
\begin{equation*}
\bigg|\bigg|\sum_{n=0}^{m-1} a_n f_{t_n}\bigg|\bigg|_\infty = \sup_{\alpha\in 2^\N}\bigg\{\bigg|\sum_{n=0}^{m-1} a_n \chi_{V_n}(\alpha)\bigg|\bigg\} = \sup_{F\in\mathcal F_m}\bigg\{\bigg|\sum_{n\in F} a_n\bigg|\bigg\}
\end{equation*}
for some finite collection $\mathcal F_m\subseteq\mathcal P([0..m-1])$ depending on $m$. Estimating the operator norm of $S_m$ can therefore be done by estimating the operator norm of an operator 
$T_\mathcal{F} : \ell^p_m\rightarrow\ell^{\infty}_{b_m}$ of the form $T_\mathcal{F}(a_0,\ldots, a_{m-1}) = (\sum_{n\in F_0} a_n, \ldots, \sum_{n\in F_{b_m-1}} a_n)$, where $F_n\in\mathcal F$ and $b_m = |\mathcal F|$.

\begin{lemma}\label{ptoinf}
Let $\mathcal F\subseteq\mathcal P([0..m-1])$ be nonempty and define $T_\mathcal{F}\in B(\ell_m^p,\ell_{b_m}^\infty)$ by 
\begin{equation*}
T_\mathcal{F}(a_0,\ldots, a_{m-1}) = \bigg(\sum_{n\in F_0} a_n, \ldots, \sum_{n\in F_{b_m - 1}} a_n\bigg)
\end{equation*} 
Let $k = \sup_{F\in\mathcal{F}}\{|F|\}$ and $q$ be the H\"older conjugate of $p$. Then $||
T_\mathcal{F}|| = k^\frac{1}{q}$.
\end{lemma}

\begin{proof}
By H\"older's inequality, we have
\begin{align*}
||T_\mathcal{F}((a_0,\ldots, a_{m-1}))||_\infty = \sup_{F\in\mathcal F}\bigg\{\bigg|\sum_{n\in F} a_n\bigg|\bigg\}\leq &\sup_{F\in\mathcal F}\bigg\{\sum_{n\in F}| a_n|\bigg\}\leq \sup_{F\in\mathcal F}\{|F|^\frac{1}{q}\:||a||_p\}= k^\frac{1}{q}||a||_p
\end{align*}
Sharpness of the estimate comes from choosing $G\in\mathcal F$ with $|G| = \sup_{F\in\mathcal F}\{|F|\}$ and taking $a = \chi_G\in\ell^p_m$.
\end{proof}

\begin{theorem}\label{basicvectorestimate}
Let $S_m$ be as above and $m\geq 1$.
\begin{enumerate}
    \item If $m = 1$, then $||S_m|| = 1$.
    \item If $m\geq 2$, then $||S_m|| = \sqrt{2+\floor{\log_2(m-1)}}$.
\end{enumerate}
\end{theorem}

\begin{proof}
When $m = 1$, $S_1$ is a map between $\ell_1^2$ and $\overline{\spn\{f_{t_0}\}}$, and it is easy to verify by hand that $||S_1|| = 1$ since $f_{t_0}$ is the constant function 1. When $m\geq 2$, by Lemma \ref{ptoinf}, we can determine the value of $||S_m||$ by finding $\sup_{F\in\mathcal F_m}\{|F|\}$ in terms of $m$ for a certain choice of $\mathcal F_m$. For each $a = (a_0,\ldots, a_{m-1})$, each entry in $S_m(a)$ is of the form $\sum_{n=0}^{m-1} a_n\chi_{V_n}$, and two sets $V_{n}$, $V_{n'}$ intersect if and only if one is a subset of the other by definition of the basis of $C(2^\N)$. Hence any entry in $S_m(a)$ must be the sum over a set $F\subseteq[0..m-1]$ for which there is an $\alpha\in 2^\N$ such that $\alpha\in V_n$ for every $n\in F$. The set $\mathcal{F}_m$ is the collection of all such $F$. By the properties of the enumeration function $h$ defined in Section 2, the largest such $F$ consists of the integers corresponding to $\varnothing$ and the finite strings $t_n$ consisting of all 0s. These occur when $n=0$ or $n = 2^k$ for some $k\in\N$, respectively. Since the maximum index among the $f_{t_n}$ under consideration is $m-1$, we have
\begin{equation*}
|F| = 1+|\{k\in\N\mid 2^k\leq m-1\}| = 1+(1+\floor{\log_2(m-1)})
\end{equation*}
which yields the stated result after applying Lemma \ref{ptoinf}.
\end{proof}

This result can be used to improve the bound on $m$ when $(e_0,\ldots, e_{m-1})\in\textbf{T}(C) = \textbf{T}_{NU}(\ell^2,(f_{t_n}),C)$, where $(e_n)$ is the standard basis of $\ell^2$. Given $C\geq 1$, let $n\in\N$ be least with $n>C^2-2$. Set $m_0 = 2^n +1$ so that $m_0\geq 2$. For $m\geq m_0$, we have
\begin{equation*}
||S_m|| = \sqrt{2+\floor{\log_2(m-1)}} \geq \sqrt{2+\floor{\log_2(m_0-1)}} = \sqrt{2+n} > C
\end{equation*}
so that $(e_0,\ldots, e_{m-1})\notin\textbf{T}(C)$. It follows that if $(e_0,\ldots, e_{m-1})\in\textbf{T}(C)$, then $m < 2^{C^2-1}+1$.

As expected, since we can unitarily transform the standard basis of $\ell^2$ into any other orthonormal basis and the defining condition of $\textbf{T}(C)$ is invariant under unitary transformations, we have the following result.

\begin{corollary}\label{reorder}
Let $\textbf{T}(C)$ be defined as above and $(e_n)$ be the standard basis of $\ell^2$. Fix $m\geq 1$ and let $(u_n)_{n=0}^{m-1}$ be an orthonormal set in $\ell^2$. Then $(u_n)_{n=0}^{m-1}\in \textbf{T}(C)$ if and only if $(e_n)_{n=0}^{m-1}\in \textbf{T}(C)$. 

 In particular, every string $s\in \textbf{T}(C)$ consisting of orthonormal vectors satisfies $|s| < 2^{C^2-1}+1$.
\end{corollary}

By loosening our grasp on the uniform bound, we can obtain an analogous improved estimate for strings consisting of pairwise-orthogonal vectors (not necessarily orthonormal).

\begin{corollary}\label{boundedorthogonal}
Suppose that $(u_0,\ldots, u_{m-1})\in \textbf{T}(C)$ is a family of pairwise-orthogonal vectors. Set $A = \inf_{n<m}\{||u_n||_2\}$, $B = \sup_{n < m} \{||u_n||_2\}$, and $C' = \max\{C B, \frac{C}{A}\}$. Then $m < 2^{(C')^2-1}+1$. 

In particular, for every string $s\in \textbf{T}(C)$ consisting of pairwise-orthogonal vectors, $|s| < 2^{C^4-1}+1$.
\end{corollary}

\begin{proof}
Without loss of generality (Corollary \ref{reorder}), we may assume that $u_n = b_n e_n$ for some $b_n\in\mathbb R\setminus\{0\}$ (so that $b_n = ||u_n||$ for each $n$). Then we have the inequalities
\begin{equation*}
A\bigg|\bigg|\sum_{n=0}^{m-1} a_n e_n\bigg|\bigg|_2 \leq \bigg|\bigg|\sum_{n=0}^{m-1} a_n b_n e_n\bigg|\bigg|_2 \leq B\bigg|\bigg|\sum_{n=0}^{m-1} a_n e_n\bigg|\bigg|_2
\end{equation*}
for every $(a_n)_{n=0}^{m-1}\subseteq \R$. Combining this with the fact that $(u_0,\ldots, u_{m-1})\in \textbf{T}(C)$ by assumption, we obtain further inequalities
\begin{equation*}
\frac{1}{C B}\bigg|\bigg|\sum_{n=0}^{m-1} a_n f_{t_n}\bigg|\bigg|_\infty \leq \bigg|\bigg|\sum_{n=0}^{m-1} a_n e_n\bigg|\bigg|_2 \leq \frac{C}{A} \bigg|\bigg| \sum_{n=0}^{m-1} a_n f_{t_n}\bigg|\bigg|_\infty
\end{equation*}
which implies that
\begin{equation*}
\frac{1}{C'} \bigg|\bigg|\sum_{n=0}^{m-1} a_n f_{t_n}\bigg|\bigg|_\infty \leq \bigg|\bigg|\sum_{n=0}^{m-1} a_n e_n\bigg|\bigg|_2 \leq C' \bigg|\bigg| \sum_{n=0}^{m-1} a_n f_{t_n}\bigg|\bigg|_\infty
\end{equation*}
By Corollary \ref{reorder}, $m\leq 2^{(C')^2-1}+1$, giving us the desired result.

For the last statement, suppose $(u_0,\ldots, u_{m-1})\in\textbf{T}(C)$. By definition, this means that 
\begin{equation*}
\frac{1}{C} \bigg|\bigg|\sum_{n=0}^{m-1} a_n f_{t_n}\bigg|\bigg|_\infty \leq \bigg|\bigg|\sum_{n=0}^{m-1} a_n u_n\bigg|\bigg|_2 \leq C \bigg|\bigg| \sum_{n=0}^{m-1} a_n f_{t_n}\bigg|\bigg|_\infty
\end{equation*}
for every $(a_n)_{n=0}^{m-1}\subseteq \R$. Given $n$, take $a_\ell = 0$ for $\ell\neq n$ and $a_n = 1$. We obtain
\begin{equation*}
\frac{1}{C} ||f_{t_n}||_\infty \leq ||u_n||_2 \leq C ||f_{t_n}||_\infty
\end{equation*}
or equivalently
\begin{equation*}
\frac{1}{C}\leq ||u_n||_2 \leq C
\end{equation*}
Hence $\frac{1}{C}\leq A \leq B\leq C$ so that $C^2\geq \max\{C B,\frac{C}{A}\}$, which proves the claim.
\end{proof}

These uniform bounds apply to $\textbf{T}_{NU}(\ell^2, (f_{t_n}), C)$ where all vectors are accessible. The coded trees obtained in steps 2 and 3 of the construction use the selection functions $d_n$ to construct the tree $T_{NU}(\ell^2, (f_{t_n}))$, and we can say little about the vectors $d_n(\ell^2)$ in general. However, by descending the definition chain, the above corollary does apply in the sense that if $(k,\ell_0,\ldots, \ell_{m-1})$ are such that $(d_{\ell_n}(\ell^2))_{n=0}^{m-1}\in T_{NU}(\ell^2, (f_{t_n}), k)$ and $(d_{\ell_n}(\ell^2))_{\ell=0}^{m-1}$ are pairwise-orthogonal, then $m < 2^{(C')^2-1}+1$, where $C'$ is defined as above.

\section{Conclusion}

There are many paths forward in this direction. The most obvious natural question is to ask about the order of other spaces with respect to the basis $(f_{t_n})$. Natural choices include $\ell^p$ and $L^p([0,1])$ for $p\in(1,\infty)$. Since the duals of these spaces are separable, none of these spaces can contain $C(2^\N)$, and so their $NU$ trees (with respect to $(f_{t_n})$) must have countable order. The proposition from \cite{TJ} cited in the proof of Lemma \ref{ML} suggests that for the spaces $\ell^p$ ($p\in (1,\infty)$) we should have $o(T_{NU}(\ell^p, (f_{t_n}))) = \omega + 1$ as well, but there is no proof of this at the moment.

Another line of inquiry is to ask what characteristics of a Banach space $X$ would force $X$ to satisfy $o(T_{NU}(X,(f_{t_n}))) = \omega + 1$. The key aspects of our argument for the $\ell^2$ case were a pair of important facts: first, that for any linearly independent set $(v_n)_{n=0}^{m-1}\subseteq\ell^2$, there exists an isomorphism $S\in B(\ell^2,\ell^2)$ with $Sv_n = e_n$; and second, that for every $n$ there exists an isometric isomorphism from $\overline{\spn\{e_n\}_{n=0}^{m-1}}$ to $\ell^2_m$. This allows us to easily isomorph our vectors to $\ell_m^2$ where we can apply the theory of Banach-Mazur distances. One possible method of generalizing this argument is to take a Banach space $X$ with Schauder basis $(y_n)$ and which satisfies two conditions: first, that for any linearly independent set $(x_n)_{n=0}^{m-1}\subseteq X$, there is an isomorphism $S\in B(X,X)$ with $Sx_n = y_n$ for every $n$; and second, that for every $m$ there exists an isomorphism $T_m\in B(\overline{\spn\{y_n\}_{n=0}^{m-1}},\ell_m^2)$ with $T_my_n = e_n$ for all $n$ and such that $\sup_{m\in\N}\{||T_m||, ||T_m^{-1}||\} < \infty$. We expect that these conditions would provide for a similar argument as in Lemma \ref{ML}. To find spaces satisfying the former condition requires the study of transitive norms, which is an area of mathematics still under active development (see Chapter 32 of \cite{GMZ} for a summary of results).

When working with these trees, the choice of Schauder basis of $C(2^\N)$ used to construct the trees is often immaterial. For example, in applications involving NC trees (of which NU trees are a subset), e.g. Theorem 10 in \cite{AD} and Theorem 2.29 in \cite{D}, the properties of the Schauder basis $(f_n)$ from which $T_{NC}(Y,X,(f_n))$ is constructed do not feature in the theorems at all. Even when considering results regarding the NU tree itself, e.g. Corollaries 28 and 30 in \cite{D3}, the Schauder bases are again taken to be generic. There are some associated subtleties to this practice in the context of this paper. In \cite{PS}, Pe{\l}czy{\'n}ski and Singer illustrated that for every Banach space $X$ with a Schauder basis, there are uncountably many normalized Schauder bases of $X$ which are not equivalent to the original basis in the sense described in Section 2. Argyros and Dodos proved that there is a $\boldsymbol \Pi_1^1$-rank $\varphi$ on $NU$ which satisfies $\varphi(X)\geq o(T_{NU}(X,(f_n)))$ for every $X\in NU$ and every Schauder basis $(f_n)$ of $C(2^\N)$ (cf. Theorem 10 in \cite{AD}). This implies that $o(T_{NU}(H, (f_n)))$ is bounded above by some countable ordinal $\xi$ regardless of the choice of the basis $(f_n)$ of $C(2^\N)$. However, it is possible that for a different choice of Schauder basis $(f_n)$ of $C(2^\N)$ which is not equivalent to $(f_{t_n})$, we may have $o(T_{NU}(H,(f_n))) > \omega + 1$. We conjecture that this is not the case when $(f_n)$ is required to be seminormalized. A theorem of Szarek (\cite{S80}; see also Section 4.8 of \cite{Se}) asserts that for every $p\in[1,\infty)$, there is no $p$-Besselian seminormalized basis in $C(2^\N)$. This means there is no Schauder basis $(f_n)$ in $C(2^\N)$ such that:
\begin{enumerate}
    \item $0 < \inf\{||f_n||_\infty\} \leq \sup\{||f_n||_\infty\} < \infty$ for all $n$
    \item there exists a constant $C > 0$ such that for every \textit{finite} sequence of scalars $(a_n)_{n=0}^{m-1}\subseteq\R$, we have 
\end{enumerate}
\begin{equation*}
\bigg|\bigg|\sum_{n=0}^{m-1} a_nf_n\bigg|\bigg|_\infty\geq C\bigg|\bigg|\sum_{n=0}^{m-1} a_n e_n\bigg|\bigg|_p
\end{equation*}
This theorem suggests that an analogue of Lemma \ref{ML} is true for every seminormalized Schauder basis $(f_n)$ of $C(2^\N)$ and would be an excellent starting point towards generalizing Theorem \ref{MR}. We are curious to know if this is indeed the case.

\section{Acknowledgements}

We would like to acknowledge professor Hari Bercovici for his direction about the Banach-Mazur distance between $\ell^2_n$ and $\ell^\infty_n$.

\printbibliography

\end{document}